\documentclass[11pt]{amsart}

\usepackage{color}
\usepackage{amsmath,amsthm,amssymb,amsfonts}
\usepackage{epsfig}
\usepackage{graphicx}
\usepackage[titletoc, title]{appendix}
\usepackage{bm}
\usepackage{dsfont}
\newtheorem{theorem}{Theorem}[section]
\newtheorem{lemma}[theorem]{Lemma}
\newtheorem{proposition}[theorem]{Proposition}
\newtheorem{corollary}[theorem]{Corollary}

\newtheorem{Remark}{Remark}[section]

\def\R{\mathbb R}
\def\Z{\mathbb Z}

\begin{document}

\title[On Unbounded Solutions of Ergodic Problems]{On Unbounded Solutions of Ergodic Problems in $\R^m$ for viscous Hamilton-Jacobi Equations}
\author[]{Guy Barles}
\address{
Guy Barles:
Laboratoire de Math\'ematiques et Physique Th\'eorique (UMR CNRS 7350), F\'ed\'eration Denis Poisson (FR CNRS 2964),
Universit\'e Fran\c{c}ois Rabelais Tours, Parc de Grandmont, 37200 Tours, FRANCE. {\tt Guy.Barles@lmpt.univ-tours.fr}}

\author[]{Joao Meireles}
\address{
Joao Meireles: 
Dipartimento di Matematica, Universit\`a di Padova, via Trieste 63, 35121 Padova, Italy. {\tt jhbmeireles@gmail.com}}

\begin{abstract} 
In this article we study ergodic problems in the whole space $\R^m$ for viscous Hamilton-Jacobi Equations in the case of locally Lipschitz continuous and coercive right-hand sides. We prove in particular the existence of a critical value $\lambda^*$ for which (i) the ergodic problem has solutions for all $\lambda \leq \lambda^*$, (ii) bounded from below solutions exist and are associated to $\lambda^*$, (iii) such solutions are unique (up to an additive constant). We obtain these properties without additional assumptions in the superquadratic case, while, in the subquadratic one, we assume the right-hand side to behave like a power. These results are slight generalizations of analogous results by N. Ichihara but they are proved in the present paper by partial differential equations methods, contrarily to N. Ichihara who is using a combination of pde technics with probabilistic arguments.
\end{abstract}



\maketitle
\tableofcontents

\section{Introduction}
We are interested in this article in {\em ergodic problems} for viscous Hamilton-Jacobi Equations in the whole space $\R^m$, namely to find a pair $(\lambda, \phi) \in \mathbb{R} \times C^2(\mathbb{R}^m)$ which solves
$$
	(EP) \quad \quad 
		-\frac{1}{2}\Delta \phi(y)+\frac{1}{\theta} | D \phi(y)|^{\theta}=f(y) -\lambda    \quad \text{in } \mathbb{R}^m,
$$
where $D \phi$ and $\Delta \phi$ denotes respectively the gradient and the Laplacian of $\phi:\R^m \to \R$, $\theta>1$ and $f$ is (at least) a bounded from below, locally Lipschitz continuous function. More precise assumptions are given later on.  Sometimes we write, in the sequel, $(EP)_\lambda$ instead of $(EP)$ to point out the dependence of the ergodic problem in $\lambda$. 

Our aim is to prove the following type of properties
\begin{itemize}
\item There exists a critical value $\lambda^*=\lambda^*(f)$ such that $(EP)$ has a solution for any $\lambda \leq \lambda^*$.
\item Problem $(EP)$ has bounded from below solutions.
\item Bounded from below solutions are unique (up to additive constants) and associated to $\lambda^*(f)$.
\item $\lambda^*(f)$ depends continuously and monotonically on $f$.
\end{itemize}
Clearly, one needs suitable assumptions to prove such results and we first point out that there is an important difference between the {\em superquadratic case} ($\theta \geq 2$) where the coercivity of $f$ turns out to be sufficient and the {\em subquadratic case} ($\theta < 2$) where we have to impose more restrictive assumption on the growth of $f(y)$, typically a behavior like $|y|^\alpha$ for some $\alpha > 0$.

We also point out that most of our results were already proved by Ichihara \cite{NI1, NI2,NI3,NI4} (see also Cirant \cite{Ci}) by combining partial differential equations (pde in short) and probabilistic methods, in particular using in a crucial way the ergodic measure for the uniqueness parts. Here we only use pde methods which allow to prove slightly more general results.

The main difficulty in the proofs comes from the unboundedness of solutions and the treatment of infinity, the most common framework for such ergodic problems being either the periodic case or the case of bounded domains with Neumann boundary conditions. 

To provide a complete description of the literature on ergodic problems is clearly impossible but we can indicate few articles which are milestones and which are related to the present work. We cannot avoid starting with the seminal (yet unpublished) work of Lions, Papanicolaou and Varadhan \cite{lpv} about homogenization of {\em first-order} Hamilton-Jacobi equations (in the case of coercive Hamiltonians). Indeed, solving the cell problem is the way the ``effective'' (or averaged) Hamilton-Jacobi Equation is determined, and this cell problem is an ergodic problem in the whole space but with periodic boundary conditions. This article is the starting point of the ``Weak KAM'' theory for which we refer to Fathi \cite{F12} and reference therein.

Extensions to the case of unbounded solutions for first-order equations in the non-periodic case is consider in the paper of Roquejoffre with the first author (\cite{BR}) : the existence of Lipschitz solutions is obtained but the most important part is to show that key properties of the ergodic problem, namely that it should govern the large time behavior of solutions of the associated evolution equation, may be wrong. Some other counterexamples of the ergodic behavior in the non-periodic case are also obtained in the work of Souganidis and the first author \cite{BS00a}.

For {\em second-order} equations, a systematic approach is developed in the periodic case by Souganidis and the first author \cite{bs01}, where, not only ergodic problems in the space-time framework are solved, but also the long time behavior of space-time periodic solutions of quasilinear pdes is established. Long time behavior and ergodic problems for second order pdes with Neumann boundary conditions have been
studied in the series of papers \cite{gl05,L08,BLLS08}. More recently, ergodic problems for viscous Hamilton-Jacobi equations in bounded domains, with state-constraint boundary conditions, have been treated by Tchamba in the
superquadratic case \cite{T10}, and  Porretta, Tchamba and the first author in the subquadratic case \cite{BPT10}. These last results are playing an important role in the present work since approximations by bounded domains are natural and we are going to use them.

As we already mentioned it, the results of the present work are close to the works of Ichihara \cite{NI1,NI2,NI3,NI4} where he studies the existence of ergodic constant and the connections with  stochastic ergodic control, and the recurrence/transience of optimal feedback processes. Compared to these results, we can treat more general $f$ but for some properties we have less precise results.

Our paper is organized as follows: in Section~\ref{Exist}, we prove the existence of solutions for $(EP)$, solving the two first point we mention at the beginning of the introduction. Section~\ref{Uni} is devoted to the uniqueness results, first in the superquadratic case ($\theta \geq 2$) where we can actually show that bounded from below solutions are unique (up to additive constants) for any locally Lipschitz continuous, {\em coercive} functions $f$ (without any restriction on its growth at infinity), while in the subquadratic case, it holds only when $f$ grows like $|y|^\alpha$ for some $\alpha>0$. Finally, in Section~\ref{l-star}, we show that bounded from below solutions are necessarely associated to $\lambda^*(f)$ and we provide some continuity properties of $\lambda^*(f)$ in $f$. We also study the convergence of several approximations and we establish the connections of the ergodic constant with the large time behavior of solutions of evolution equations.

\noindent{\bf Acknowledgement : }: Joao Meireles was partially supported by EU under
  the 7th Framework Programme Marie Curie Initial Training Network
  ``FP7-PEOPLE-2010-ITN'', SADCO project, GA number 264735-SADCO.

\section{General Existence Results for the Ergodic Problem and Properties of the Ergodic Constants}\label{Exist}
In this section we investigate all the properties of the ergodic problem and ergodic constants which are the same in the subquadratic and superquadratic case. Our aim is therefore to use as general assumptions on $f$ as possible, even if we have to restrict them later. In the sequel, if $\phi \in C^2(\R^m)$, we use the notation $ G[\phi]$ for the function $-\frac{1}{2} \Delta \phi + \frac{1}{\theta} |D \phi|^{\theta}-f.$

\subsection{A General Existence Result and its Main Consequences}

A first key result is the
\begin{theorem}\label{ber}
Suppose that $f \in W^{1,\infty}_{loc}(\R^m)$. If, for $\lambda \in \mathbb{R}$, $(EP)_\lambda$ has a $C^2$-subsolution, then $(EP)_\lambda$ has a solution.
\end{theorem}
\begin{proof}
Let $\psi$ be a $C^2$-subsolution of $(EP)_\lambda$.  By Theorem~A.1(b) (Appendix~A), we know that, for any $R>1$, there exists a solution $\phi_R \in C^{2}(\bar{B}_R)$ of
\begin{equation*}
G[\phi_R]= - \lambda \text{ in } B_R, \quad \phi_R=\psi \text{ on } \partial B_R,
\end{equation*}
where, here and below, $B_R:=\{x: |x|<R\}$. We now set $\hat{\phi}_R = \phi_R -\phi_R(0)$ and consider the family of solutions $\{ \hat{\phi}_R \}_{R>1}$. By Theorem~A.2, we know that, for all $0 < R'+1<R$ there exists a constant $C>0$ depending only on $R'$, $\theta$ and $m$ such that
\begin{align*}
\sup_{B_{R'}} |D \hat{\phi}_R| &\leq C(1+\sup_{B_{R'+1}}| f(y) -\lambda|^{\frac{1}{\theta}}+\sup_{B_{R'+1}} |D  f(y)  |^{\frac{1}{2\theta-1}}).
\end{align*} 
In particular, we observe that, for all $R>R'+1$, $\sup_{B_{R'}} |D \hat{\phi}_R|$ is bounded by a constant which does not depend on $R$. This gradient bound together with the fact that $\hat{\phi}_R(0)=0$ also imply a local $L^\infty$-bound depending on $R'$ but not on R. Then, using elliptic regularity, we deduce from this gradient bound that $\hat{\phi}_R$ is bounded in $C^{2, \iota}(B_{R'})$ for any $0<\iota <1$ by a constant independant of $R>R'+1$.

We point out that this last step use in an essential way the fact that $\hat{\phi}_R$ is locally uniformly bounded, an information which may not be true with the family of solutions $\{ \phi_R \}_{R>1}$ since $\phi_R (0)$ may go to infinity when $R$ tends to infinity.

We can now apply Arzela-Ascoli's Theorem not only for $\hat{\phi}_R$ but also for $D \hat{\phi}_R$ and $D^2 \hat{\phi}_R$ and conclude, up to a diagonal extraction argument, that there exist a sequence $\{R_n\}_n \to +\infty$ as $n \to +\infty$ such that $\{\phi_n \}_n := \{ \hat{\phi}_{R_n} \}_n$ converges locally in $\mathbb{R}^m$ to a function $\phi \in C^{2} (\mathbb{R}^m)$ which satisfies $G[\phi](y)=- \lambda$ in $\R^m$.
\end{proof} 

An immediate corollary of Theorem~\ref{ber} is the
\begin{corollary}\label{ins}(Infinite number of solutions) If $(\lambda_1,\phi)$ is a subsolution of $(EP)$, then there exist a solution of $(EP)$ for any $\lambda_2 \leq\lambda_1$. In particular, $(EP)$ has infinitely many solutions.
\end{corollary}

An interesting case where these results apply is the case of bounded from below functions $f$
\begin{corollary}\label{ebfb}
If $f \in W^{1,\infty}_{loc}(\R^m)$ is bounded from below, then $(EP)$ has infinitely many solutions.
\end{corollary}
\begin{proof} 
The proof is immediate remarking that, for $\lambda = \min_{\R^m} (f)$, any constant function is a subsolution of $(EP)$.
\end{proof} 
\subsection{Existence of a critical value}

The main result of this section is the

\begin{theorem}\label{exist-CV}
Assume that $f \in W^{1,\infty}_{loc}(\R^m)$ is bounded from below and set $\lambda^* := \sup \{\lambda \in \mathbb{R} | \text{ $(EP)$ has a subsolution} \}$. Then $\lambda^*$ is finite and $(EP)_{\lambda^*}$ has a $C^2$-solution.
\end{theorem}
	\begin{proof} The proof of Corollary~\ref{ebfb} implies $\lambda^* \geq \min_{\R^m} (f)$ because, for $\lambda=\min_{\R^m} (f)$, any constant function is a subsolution of $(EP)_\lambda$. To prove that $\lambda^* < + \infty$, we argue by contradiction assuming that $\lambda^*  = + \infty$. Then, there exists a sequence of solutions $\{(\lambda_k, \phi_k) \}_k$ of $(EP)$ such that $\lambda_k \to +\infty$ as $k \to +\infty$. We define $\psi_k := \lambda_k^{-\frac{1}{\theta}} \hat{\phi}_k$ where $\hat{\phi}_k = \phi_k -\phi_k(0)$. Then, by Theorem~A.2, $\hat{\psi}_k$ is locally uniformly bounded in $W^{1,\infty}$ by the same argument as in the proof of Theorem~\ref{ber} since $\{(\lambda_k, \psi_k) \}_k$ solves
$$
-\frac{1}{2} \lambda_k^{-\frac{1}{\theta^*}} \Delta \psi_k + \frac{1}{\theta} |D \psi_k|^\theta = \lambda_k^{-1} f_k -1\quad\hbox{in  }\R^m,
$$
where $1=\frac{1}{\theta}+\frac{1}{\theta^*}$ and $\lambda_k^{-\frac{1}{\theta^*}} \to 0$ as $k \to +\infty$. Then using the estimate for the solution given by Theorem A.2 and which is uniform in $k$, we can apply Ascoli-Arzela Theorem together with a diagonal extraction procedure : this shows that there exists a subsequence of the sequence $(\psi_k)_k$ which converges locally uniformly to a function $\psi \in W^{1,\infty}_{loc}(\R^m)$.

Therefore, by the stability result for viscosity solutions, $\psi$ satisfies
\begin{equation*}
\frac{1}{\theta} |D \psi |^\theta \leq -1
\end{equation*}
which is a contradiction. Thus $\lambda^*$ is finite.

To complete the proof, we have to show that $(EP)_{\lambda^*}$ has a $C^2$-solution. To do so, we choose any sequence $(\lambda_k)_k$  such that $\lambda_k < \lambda^*$ and $\lambda_k \to \lambda^*$. By the definition of $\lambda^*$ and Corollary~\ref{ins}, $(EP)_{\lambda_k}$ has a $C^2$-solution $\phi_k$. Repeating exactly the arguments of Theorem~\ref{ber} shows that $\hat{\phi}_k:=\phi_k-\phi_k(0)$ converges to a function $\phi$ that is a $C^2$-solution of $(EP)$ with $\lambda = \lambda^*$. 
\end{proof}

We conclude this section by giving immediate properties of the critical value $\lambda^*(f)$.

\begin{proposition}\label{lambda-fp}Assume that $f \in W^{1,\infty}_{loc}(\R^m)$ is bounded from below. Then \\
(i) For any $c \in \mathbb{R}$, $\lambda^*(f+c)=\lambda^*(f)+c$.\\
(ii) (Monotonicity of $\lambda^*$ with respect to $f$)
Suppose that $f_1,f_2\in W^{1,\infty}_{loc}(\R^m)$ are bounded from below. If $f_1 \leq f_2$, then $\lambda^*(f_1) \leq \lambda^*(f_2)$.\\
(iii) (Concavity of $\lambda^*$ with respect to $f$) For any $f_1,f_2\in W^{1,\infty}_{loc}(\R^m)$ which are bounded from below and for any $0\leq t \leq 1$
$$ \lambda^*(tf_1+(1-t)f_2)\geq t\lambda^*(f_1)+(1-t)\lambda^*(f_2)\; .$$
\end{proposition}

\begin{proof}We just prove (iii). If $\phi_1$ is a solution of $(EP)_\lambda$ for $f_1$ and $\lambda=\lambda^*(f_1)$ and if $\phi_2$ is a solution of $(EP)_\lambda$ for $f_2$ and $\lambda=\lambda^*(f_2)$ then by the convexity of the equation $t\phi_1+(1-t)\phi_2$ is a subsolution of $(EP)_\lambda$ for $tf_1+(1-t)f_2$ and with $t\lambda^*(f_1)+(1-t)\lambda^*(f_2)$. By the definition of $\lambda^*(tf_1+(1-t)f_2)$, we have the concavity inequality.
\end{proof}

The purpose of the next sections is to show, under different assumptions in the sub or superquadratic case, that if $f \in W^{1,\infty}_{loc}(\R^m)$ is bounded from below, then bounded from below solutions of $(EP)_\lambda$ can exist only if  $\lambda=\lambda^*$ and that these solutions actually exist and are unique up to an additive constant.

\subsection{Bounded from below solutions}

Now we turn to the existence of a bounded from below solution of $(EP)$.

\begin{theorem}\label{exist-bfb}  Assume that $f \in W^{1,\infty}_{loc}(\R^m)$ and is coercive, i.e. 
$$f(x)\to +\infty\quad\hbox{ when $|x|\to +\infty$}\; ,$$
then there exists a solution of $(EP)$ which is bounded from below.
\end{theorem}

\begin{proof} We use results of Tchamba Tabet \cite{T10} in the superquadratic case (here $\theta>2$) or Porretta, Tchamba and the first author\cite{BPT10} in the subquadratic case (here $\theta\leq 2$). For $R>0$, we consider the ergodic problem
$$
-\frac{1}{2} \Delta \phi^{R} +\frac{1}{\theta} |D\phi^{R}|^\theta =f -\lambda_R \text{ in } B_R\; ,
$$
with state constraint boundary conditions: we recall that $\phi^{R}(x) \to +\infty$ when $x$ tend to $\partial B_R$ in the subquadratic case ($\theta \leq 2$) while the boundary condition just means that a minimum point of $\phi^{R}-\chi$ cannot be on the boundary if $\chi$ is a smooth function in $\R^m$. We may also assume that $\phi^{R}(0)=0$. By using results or arguments of \cite{T10,BPT10}, it is easy to see that, if $R\leq R'$
$$\lambda_R\geq \lambda_{R'} \geq \lambda^*(f)\; ,$$
and therefore the $\lambda_R$ are uniformly bounded for $R\geq 1$.

Before letting $R$ tend to $+\infty$, we show that, at least for large $R$, the minimum of $\phi^{R}$ is achieved in a fixed compact subset of $\R^m$ (independent of $R$). To do so, we consider $\min_{\bar{B}_R}(\phi^{R}(x))$ which is achieved at a point denoted by $\bar x_R$.  At the point $\bar x_R$ which is in $B_R$ by the properties of $\phi^R$ we recalled above, we have
$$ 0 \geq f (\bar x_R) -\lambda_R \; .$$
i.e $f (\bar x_R) \leq \lambda_R$. The $\lambda_R$ being uniformly bounded, this means that the $\bar x_R$ remain uniformly bounded, independently of $R$ since $f$ is coercive.

Next using the same arguments as in the proof of Theorem~\ref{ber}, it is easy to show that, up to a subsequence, $\phi^{R}$ converges in $C^2(\R^m) $ to a function $\phi$ which solves
$$
-\frac{1}{2} \Delta \phi +\frac{1}{\theta} |D\phi |^\theta =f -\lambda \text{ in } \R^m\; ,
$$
with $\lambda = \lim_R (\lambda_R)$. 

Finally we have the inequality
$$ \phi^{R}(\bar x_R) \leq \phi^{R}(x) \quad \hbox{for any  }x \in \bar{B}_R\; ,$$
and up to a subsequence we can assume that we have both a local uniform convergence for $\phi_R$ and $\bar x_R \to \bar x$. By passing to the limit we obtain
$$ \phi(\bar x) \leq \phi(x) \quad \hbox{for any  }x \in \R^m\; ,$$
showing that $\phi$ is bounded from below.
\end{proof}

\section{Uniqueness results for the ergodic problem}\label{Uni}

Unfortunately the proof of these uniqueness results are rather different in the subquadratic case ($\theta< 2$) or superquadratic case ($\theta \geq 2$). This is why we have to consider them separately.

\subsection{The Superquadratic Case}

The first uniqueness result for bounded from below solutions of $(EP)$ in the superquadratic case is the
\begin{theorem}\label{uni-sup} Assume that $f \in W^{1,\infty}_{loc}(\R^m)$ is coercive and $\theta \geq 2$. If $(\lambda_1, \phi)$ and $(\lambda_2, \psi)$ are two solutions of $(EP)$ such that $\phi$ and $\psi$ are bounded from below, then $\lambda_1=\lambda_2$ and there exists a constant $C \in \R$ such that $\phi=\psi+C$. 
\end{theorem}

Before providing the proof of this result, we have to state and prove an intermediate result. To do so, we introduce a key ingredient in the superquadratic case, namely the Hopf-Cole transformation $z=-e^{-\phi}$. If  $(\lambda, \phi)$ is a solution (resp. subsolution, supersolution) of $(EP)$ then $z(y)=-e^{-\phi(y)}$ is a solution (resp. subsolution, supersolution) of
\begin{equation}\label{T1}
-\frac{1}{2} \Delta z + N(y,z,D z) =0\quad\hbox{in  }\R^m\; ,
\end{equation}
where $N(y,z,Dz):= z \big( \frac{1}{2} \big|\frac{D z}{z} \big|^2-\frac{1}{\theta}\big| \frac{D z}{z} \big|^\theta+f-\lambda \big)$.\\

This transformation allows to take care of the problem at infinity in the following way.

\begin{proposition}\label{comp-ext} Assume that $f \in W^{1,\infty}_{loc}(\R^m)$ is coercive and $\theta \geq 2$. If $\phi$, $\psi$ are respectively a bounded from below subsolution and a bounded from below supersolution of $(EP)_\lambda$, there exists $R > 0$ depending only on $f$ and $\lambda$ such that: if $z_1=-e^{-\phi}$ and $z_2=-e^{-\psi}$ are such that $z_1 \leq z_2$ on $\partial B_R$ we have 
\begin{equation*}
z_1(y) \leq z_2(y) \text{ for all } y \in B^c_R.
\end{equation*}
\end{proposition}
\begin{proof}We first notice that since $\phi$, $\psi$ are bounded from below, the functions $z_1,z_2$ are bounded. We argue by contradiction assuming that $M_R:=\sup_{B^c_R} (z_1(y)-z_2(y))>0$ and for $0<\delta \ll 1$, we consider $M_R^\delta :=\sup_{B^c_R} (z_1(y)-z_2(y)-\delta |y|^2)$. We know that $M_R^\delta \to M_R$ as $\delta \to 0$ and therefore $M_R^\delta >0$ if $\delta$ is small enough.

Since $(z_1-z_2)(y)-\delta |y|^2 \to -\infty$ as $|y| \to +\infty$, there exists a maximum point $y^*_\delta$ in $\bar{B}^c_R$. We cannot have $y^*_\delta \in \partial B_R$ because, by our hypothesis, $z_1(y^*_\delta)-z_2(y^*_\delta) \leq 0$ and we would get that
\begin{equation*}
M_R^\delta= z_1(y^*_\delta)-z_2(y^*_\delta)-\delta|y^*_\delta|^2 \leq 0.
\end{equation*}
Therefore $y^*_\delta$ is not on the boundary and we have
\begin{equation*}
D (z_1-z_2) (y^*_\delta) = 2 \delta y^*_\delta
\quad\hbox{and}\quad
\Delta (z_1 - z_2)(y^*_\delta) \leq 2 \delta m.
\end{equation*}
On the other hand
\begin{equation*}
-\frac{1}{2} \Delta z_1 (y^*_\delta) +N(y^*_\delta,z_1(y^*_\delta),D_y z_1 (y^*_\delta)) \leq 0
\end{equation*}
and
\begin{equation*}
-\frac{1}{2} \Delta z_2 (y^*_\delta)+N(y^*_\delta,z_2(y^*_\delta),D_y z_2 (y^*_\delta)) \geq 0 .
\end{equation*}
Subtracting the second from the first inequality, we  arrive at
\begin{equation*}
-\frac{1}{2} \Delta (z_1-z_2)(y^*_\delta)+N(y^*_\delta, z_1(y^*_\delta) , D z_1(y^*_\delta) ) - N(y^*_\delta, z_2(y^*_\delta) , D z_2(y^*_\delta) ) \leq 0,
\end{equation*}
i.e.,
\begin{equation} \label{T4}
N(y^*_\delta, z_1(y^*_\delta) , D z_2(y^*_\delta)+2 \delta y^*_\delta ) - N(y^*_\delta, z_2(y^*_\delta) , D z_2(y^*_\delta)) \leq \delta m.
\end{equation}\\
Let $t \in [0,1]$ and define $X(t):= t z_1(y^*_\delta) +(1-t) z_2(y^*_\delta)$, $Y(t):=D z_2(y^*_\delta)+2 t \delta y^*_\delta$ and $h(t) := N(y^*_\delta, X(t) , Y(t))$.
Then,
\begin{align*}
&N(y^*_\delta, z_1(y^*_\delta) , D z_2(y^*_\delta)+2 \delta y^*_\delta ) - N(y^*_\delta, z_2(y^*_\delta) , D z_2(y^*_\delta)) \\
&= N(y^*_\delta, X(1), Y(1))-N(y^*_\delta, X(0), Y(0))\\
&=h(1)-h(0)\\
&=\int^1_0 h'(t)dt
\end{align*}
and
\begin{equation*}
h'(t)=\frac{\partial N}{\partial X} (y^*_\delta, X(t),Y(t)) (M_R^\delta+\delta |y^*_\delta|^2) +\frac{\partial N}{\partial Y} (y^*_\delta, X(t),Y(t)) \cdot (2 \delta y^*_\delta).
\end{equation*}
Thus inequality $(\ref{T4})$ can be re-written as
\begin{equation} \label{T5}
\int_0^1 \big[ \frac{\partial N}{\partial X} (y^*_\delta, X(t),Y(t)) (M_R^\delta+\delta |y^*_\delta|^2) +\frac{\partial N}{\partial Y} (y^*_\delta, X(t),Y(t)) \cdot (2 \delta y^*_\delta) \big]  \leq \delta m.
\end{equation}
Set $Q:=\frac{Y}{X}$. We have
 \begin{align*}
\frac{\partial N}{\partial X} (y,X,Y) = -\frac{1}{2} \big| Q \big|^2 +(1-\frac{1}{\theta}) \big| Q \big|^\theta + \big( f(y) -\lambda \big)
\end{align*}
and
\begin{align*}
\frac{\partial N}{\partial Y} (y,X,Y) &= X \big( \frac{Y}{X^2} - \frac{|Y|^{\theta-2}}{|X|^\theta} Y \big) = Q- \frac{X |Y|^{\theta-2}}{|X|^\theta} Y \\
& =Q+ \frac{|Y|^{\theta-2}}{|X|^{\theta-1}} Y 
\end{align*}
where we used $|X|=-X$ since $z_1,z_2 \leq 0$.
\begin{itemize}
\item Case $\theta=2$
\end{itemize}
In this case, $\frac{\partial N}{\partial X} (y,X,Y)= f(y)-\lambda$ and $\frac{\partial N}{\partial Y} (y,X,Y)=0$. Then, $(\ref{T5})$ is reduced to
\begin{equation*} 
(f(y^*_\delta)-\lambda) (M_\delta+\delta|y^*_\delta|^2) \leq \delta m
\end{equation*}
a contradiction if $\delta$ is small enough because we can choose $R$ large enough in order to have $f(y^*_\delta)-\lambda \geq 1$ and $M_R^\delta \geq M_R/2$.
\begin{itemize}
\item Case $\theta>2$
\end{itemize}
First we notice, by Cauchy-Schwarz inequality, that
\begin{align*} 
 \frac{\partial N}{\partial Y} \cdot (2 \delta y^*_\delta ) \geq -  2 \delta \big| \frac{\partial N}{\partial Y}  \big| \big| y^*_\delta \big|.
\end{align*}
Therefore $(\ref{T5})$ implies
\begin{align} \label{T6}
\int_0^1 \big[ \frac{\partial N}{\partial X} (y^*_\delta, X(t),Y(t)) (M_R^\delta+\delta |y^*_\delta|^2) -  2 \delta \big| \frac{\partial N}{\partial Y} (y^*_\delta, X(t),Y(t)) \big| \big| y^*_\delta \big| \big]  \leq \delta m.
 \end{align}\\
It is easy to see that, for $R$ large enough, since $\displaystyle 1-\frac{1}{\theta} >1/2$
\begin{align*}
\frac{\partial N}{\partial X} (y,X,Y) &=-\frac{1}{2} |Q|^2+ (1-\frac{1}{\theta}) |Q|^\theta + f(y)-\lambda \\
& \geq \frac{1}{3} (1+ |Q|^\theta)
\end{align*}
 for all $|y| \geq R$. Indeed, one can combine Young's inequality for the $ |Q|^2$-term with the fact that  $f(y)-\lambda$ can be taken as large as we wish by choosing $R$ large enough. 
 
 On the other hand, we have
\begin{align*}
\big| \frac{\partial N}{\partial Y} (y^*_\delta, X,Y) \big| \leq \big| Q \big| + \big| Q \big|^{\theta-1} \leq 2 \big(1+\big| Q \big|^{\theta} \big).
\end{align*}
Indeed, the first inequality follows immediately from the computation of $\frac{\partial N}{\partial Y} (y, X,Y)$ while the second comes by noticing that if $\big| Q \big| \leq 1$ then $\big| Q \big| + \big| Q \big|^{\theta-1} \leq 2$ and if $\big| Q \big| > 1$ then $\big| Q \big| < \big| Q \big|^{\theta-1} < \big| Q \big|^\theta$ because $\theta>2$. \\ \\
Hence, combining these two properties, we obtain
\begin{align*}
\big| \frac{\partial N}{\partial Y} (y^*_\delta, X,Y) \big|  \leq 6 \frac{\partial N}{\partial X} (y^*_\delta, X,Y).
\end{align*}
Then $(\ref{T6})$ implies
\begin{align*} 
&\int_0^1 \big[ \frac{\partial N}{\partial X} (y^*_\delta, X(t),Y(t)) (M_R^\delta+\delta |y^*_\delta|^2) -  12 \delta \frac{\partial N}{\partial X} (y^*_\delta, X(t),Y(t)) \big| y^*_\delta \big| \big]  \leq \delta m 
\end{align*}
\begin{align} \label{T7}
 \int_0^1  \frac{\partial N}{\partial X} (y^*_\delta, X(t),Y(t)) (M_R^\delta + \delta |y^*_\delta|^2 -  12 \delta | y^*_\delta | )  \leq \delta m .
 \end{align}
Given that $M_R^\delta \to M_R$, $\delta |y^*_\delta|^2 \to 0$ and $\delta | y^*_\delta | \to 0$ as $\delta \to 0$, we can see that
\begin{equation*}
M_R^\delta + \delta |y^*_{\delta}|^2 -  12 \delta \big| y^*_{\delta} \big| \to  M_R>0 \text{ when $\delta \to 0$}.
\end{equation*}
For $ \delta$ small enough, this gives a contradiction because $\frac{\partial N}{\partial X} (y^*_{\delta_0}, X(t),Y(t))\geq 1/3$. \end{proof}

\begin{Remark}:  If $\phi, \psi \to +\infty$ at infinity then $z_1=-e^{-\phi}$ and $z_2=-e^{-\psi}$ tend to $0$ as $|y| \to \infty$. Consequently $\max_{y \in B^c_R} \{ z_1(y) - z_2 (y) \}$ exists and we do not need to use the penalisation term $\delta|y|^2$.
\end{Remark}

\begin{proof}[Proof of Theorem~\ref{uni-sup}]
Let $R>1$ be as in Proposition~\ref{comp-ext} and suppose that $\lambda_1 \geq \lambda_2$ (otherwise we exchange the roles of $\phi$ and $\psi$ in the argument). Observe that this implies that $(\lambda_1, \phi)$ is a subsolution of $(EP)_{\lambda_2}$. We will first prove that $\phi=\psi$ in $\mathbb{R}^m$ and then conclude that $\lambda_1=\lambda_2$.

Adding constants to $\phi$ and $\psi$ we may assume that $\max_{\partial B_R } (\phi-\psi)=0$ and, if $z_1=-e^{-\phi}$ and $z_2=-e^{-\psi}$, we have $z_1 \leq z_2$ on $\partial B_R$.  We deduce from Proposition~\ref{comp-ext} that
$$
z_1(y) \leq z_2(y) \text{ for all } y \in B^c_R.
$$
and this gives $\phi (y) \leq \psi (y) \text{ for all } y \in B^c_R$.

On the other hand, in $B_R$, applying the Strong Maximum Principle, we have
\begin{equation} \label{Prop4.2-4}
\max_{\bar B_R} (\phi-\psi) = \max_{\partial B_R}  (\phi-\psi)=0. 
\end{equation}
But this means that the global maximum of $\phi-\psi$ in $\R^m$ is achieved at a point of $\partial B_R$ and applying again the Strong Maximum Principle, we deduce that $\phi=\psi$ in $\R^m$ (since the max is actually $0$), and therefore $\lambda_1=\lambda_2$.
\end{proof}

\subsection{The Subquadratic case}

Unfortunately the result in the subquadratic case is not as general as in the superquadratic one : we have to impose some restriction on the growth of $f$ at infinity.

\subsection{Existence for right-hand sides with polynomial growth}

We assume, in this section, that $f$ satisfies
\begin{itemize}
	\item[(H0)] $f \in W^{1,\infty}_{\text{loc}}(\mathbb{R}^m)$ is bounded from below and there exists $\alpha, f_0 > 0$ such that, for all $y \in \mathbb{R}^m$
$$
|Df(y)| \leq f_0 (1+|y|^{\alpha-1})\quad \hbox {if  $\alpha \geq 1$}\quad \hbox{or  }\quad |Df(y)| \leq f_0 \quad \hbox {if  $\alpha < 1$}\; .
$$
\end{itemize} 

In this case, we have a precise estimate on the growth of the solutions of $(EP)$.
\begin{proposition}\label{bainf1} Let $(\lambda, \phi)$ be a solution of (EP). Then there exists a constant $K>0$ such that
\begin{equation*}
 | D \phi (y)| \leq K (1+| y|^{\gamma-1}), \quad  | \phi (y)| \leq K(1+| y|^\gamma), \quad y \in \mathbb{R}^m,
\end{equation*}
where $\displaystyle \gamma = \frac{\alpha}{\theta}+1$.
\end{proposition} 
\begin{proof}
From Corollary A.3 in Appendix A, we have for all $r>0$ that there exists a constant $C>0$ such that
\begin{equation*}
\sup_{B_r} |D \phi(y)| \leq C(1+\sup_{B_{r+1}}|f(y)-\lambda|^{\frac{1}{\theta}}+\sup_{B_{r+1}}|Df(y)|^{\frac{1}{2\theta-1}}).
\end{equation*}
Hypothesis (H0) implies that $f$ grows at most like $|y|^\alpha$. Since $\frac{\alpha}{\theta}=\gamma-1$ and $\frac{\alpha-1}{2\theta-1} < \gamma-1$ if $\alpha \geq 1$ or $|Df(y)|$ is bounded if $\alpha <1$, we see that
\begin{equation*}
\sup_{B_r} |D \phi| \leq C( 1+ r^{\gamma-1} )
\end{equation*}
for some (possibly different) constant $C>0$. From this inequality we deduce the first estimate of this proposition and the second one follows.
\end{proof}

The next case we investigate is the case when $f$ actually grows like $ |y|^\alpha$, namely satisfies
\begin{itemize}
\item[(H1) ] $f \in W^{1,\infty}_{\text{loc}}(\mathbb{R}^m)$ and there exists $f_0 ,\alpha > 0$ such that, for all $y\in \R^m$
$$
f^{-1}_0 (|y|^\alpha+1) \leq f(y) \leq f_0 (|y|^\alpha+1)
$$
\end{itemize}
It is worth pointing out that this assumption may not seem to be as general as it could be, because of the lower bound where we could have subtracted a constant, but, since we are interested in solving the ergodic problem, this translation by a constant has no effect as it can be seen in Proposition~\ref{lambda-fp} (i).

The first result we have is the
\begin{proposition}\label{batinf}Assume that $f$ satisfies (H0)-(H1) and that $\theta >1$. If $\phi$ is a bounded from below solution of $(EP)$ then there exists $c>0$ such that $\phi(y) \geq c|y|^\gamma-c^{-1}$ where $\gamma=\frac{\alpha}{\theta}+1$.
\end{proposition}
\begin{proof} Adding constants to $\phi$ if necessary we may assume that $\phi \geq 0$. We already know that $\phi$ satisfies, for all $y\in \R^m$
\begin{equation*}
|D \phi(y)| \leq K(1+|y|^{\gamma-1}) \quad \text{(Proposition~\ref{bainf1})}  
 \end{equation*}
and
\begin{equation*}
|\phi(y)|\leq K(1+|y|^\gamma) \quad \text{(consequence of the previous estimate)} 
\end{equation*}
for some constant $K$.

We argue by contradiction assuming that there exists a sequence $|y_\epsilon| \to +\infty$ such that $\frac{\phi(y_\epsilon)}{|y_\epsilon|^\gamma} \to 0$. We set $\Gamma_\epsilon=\frac{|y_\epsilon|}{2}$ and we introduce
\begin{equation*}
v_\epsilon(y)=\frac{\phi(y_\epsilon+\Gamma_\epsilon y)}{{\Gamma_\epsilon}^\gamma} \text{ for } |y| \leq 1.
\end{equation*}
Because of the above estimates on $\phi$, we have $|v_\epsilon|$, $|Dv_\epsilon|$ uniformly bounded and $v_\epsilon$ satisfies
\begin{equation*}
-\frac{1}{2}\Gamma^{\gamma-2-\alpha}_\epsilon \Delta v_\epsilon +\frac{1}{\theta} |D v_\epsilon|^\theta = \Gamma^{-\alpha}_\epsilon \big( f(y_\epsilon+\Gamma_\epsilon y)-\lambda \big) \text{ in } B_1.
\end{equation*}
Then we notice 
\begin{equation*}
\gamma-2-\alpha=\frac{\alpha}{\theta}-1-\alpha=\alpha(\frac{1}{\theta}-1)-1 < 0
\end{equation*}
and therefore $\Gamma^{\gamma-2-\alpha}_\epsilon \to 0$ as $\epsilon \to 0$.
\begin{equation*}
f(y_\epsilon+\Gamma_\epsilon y) \geq f_0^{-1} (|\Gamma_\epsilon|^\alpha+1)\;\text{ for } |y| \leq 1,
\end{equation*}
since $|y_\epsilon+\Gamma_\epsilon y| \geq |y_\epsilon|-\Gamma_\epsilon \geq \Gamma_\epsilon$.

Since $(v_\epsilon)$ is precompact in $C(\bar{B}_1)$, we can apply Ascoli's Theorem and pass to the limit in the viscosity sense: if $v_\epsilon \to v$ then
\begin{equation*}
\frac{1}{\theta} |Dv|^\theta \geq f_0^{-1} \text{ in } B_1
\end{equation*}
and
\begin{equation*}
v \geq 0 \text{ on } \partial B_1 \text{ since $\phi \geq 0$}
\end{equation*}
therefore $v$ is a supersolution of the equation $\frac{1}{\theta} |Du|^\theta = f_0^{-1}$ with an homogeneous boundary condition for which the unique solution is $(\theta f_0^{-1})^{1/\theta} d(y, \partial B(0,1))$. By comparison principle for the eikonal equation the supersolution is above the solution. Then,  $v(y) \geq (\theta f_0^{-1})^\frac{1}{\theta} d(y,\partial B_1)$ and $v(0) \geq (\theta f_0^{-1})^{\frac{1}{\theta}}$.

But this is a contradiction since $v_\epsilon (0)=2^\gamma \frac{\phi(y_\epsilon)}{|y_\epsilon|^\gamma} \to 0$ by our hypothesis.
\end{proof}

Gathering the results of Proposition~\ref{batinf} and Theorem~\ref{exist-bfb}, we have the
\begin{corollary}\label{exist-gamma}Assume that $f$ satisfies (H0)-(H1) and $\theta>1$.
There exists a solution $(\lambda,\phi)$ of $(EP)$ such that $\phi$ belongs to $
\Phi_\gamma$, the set of functions $v \in C^2 (\mathbb{R}^m)$ for which there exists $c>1$ such that, for all $y \in \R^m$
$$c^{-1} |y|^\gamma -c \leq v(y)\leq c (1+|y|^\gamma) \; .
$$
\end{corollary}

We point out that the results of Proposition~\ref{batinf} and Corollary~\ref{exist-gamma} are valid for any $\theta >1$ and not only in the subquadratic case, even if we use them only in the subquadratic case.

\subsection{Uniqueness} The proof is very different in the subquadratic case: while we are using the transformation $\phi \to -e^{-\phi}$ in the superquadratic case {\em for both sub and supersolutions}, we are going to use here the transformation $\phi \to \phi^q$ where $q>1$ is very close to 1, {\em but only for the supersolution}. If we assume both the sub and supersolution to be in $\Phi_\gamma$, this ensures that the supersolution grows faster at infinity, solving the problem at infinity.

The key result is the
\begin{lemma}\label{lem-subq}
Let $(\lambda, \phi)$ be a supersolution of $(EP)$ such that $\phi \in \Phi_\gamma$ and $\phi \geq 1$. Then, there exists $R>1$ and $q_0 > 1$ such that for all $q \in (1, q_0)$, $(\lambda,\phi^q)$ is a strict supersolution of $(EP)$ in $B^c_R $. 
\end{lemma}
\begin{proof}
We wish to prove that, there exists $R>1$ and $q_0 > 1$ such that for all $q \in (1, q_0)$ 
\begin{equation*}
Q(y)>0 \quad \text{for all } y \in B^c_R 
\end{equation*}
where $Q(y):=-\frac{1}{2} \Delta \phi^q (y) + \frac{1}{\theta} |D \phi^q (y)|^\theta - (f(y)-\lambda)$.
We have,
\begin{equation*}
D \phi^q = q \phi^{q-1} D \phi
\end{equation*}
and
\begin{equation*}
\Delta \phi^q = q (q-1) \phi^{q-2} |D \phi|^2+q\phi^{q-1} \Delta \phi.
\end{equation*}
Then $Q$ becomes
\begin{equation*}
Q= -\frac{1}{2} \big(  q (q-1) \phi^{q-2} |D \phi|^2+q\phi^{q-1} \Delta \phi \big) + \frac{1}{\theta} \big| q \phi^{q-1} D \phi \big|^{\theta}  -( f-\lambda). 
\end{equation*}
By adding and subtracting $\frac{1}{\theta} q \phi^{q-1} |D \phi|^\theta$, using that $\phi$ is a supersolution of $(EP)$ and noticing that $(1-q\phi^{q-1})=[(1-\phi^{q-1})-(q-1)\phi^{q-1}]$, we arrive at
\begin{align*}
Q 
 &\geq-\frac{1}{2} q (q-1) \phi^{q-2} |D \phi|^2 +\frac{1}{\theta} \big( q^\theta \phi^{\theta(q-1)} - q\phi^{q-1} \big) |D \phi|^\theta  \\ & \quad - [(1-\phi^{q-1})-(q-1)\phi^{q-1}] (f-\lambda)\\
 &\geq-\frac{1}{2} q (q-1) \phi^{q-2} |D \phi|^2 +\frac{1}{\theta} \big( q^\theta \phi^{\theta(q-1)} - q\phi^{q-1} \big) |D \phi|^\theta - (1-\phi^{q-1})(f-\lambda)  \\ 
&\quad + (q-1)\phi^{q-1} (f-\lambda).
\end{align*}
But
\begin{equation*}
\frac{1}{\theta} \big( q^\theta \phi^{\theta(q-1)} - q\phi^{q-1} \big) |D \phi|^\theta \geq 0 \text{ and }  - (1-\phi^{q-1})(f-\lambda) \geq 0
\end{equation*}
because $q>1$, $\phi \geq 1$ and for $R$ large enough we have $f-\lambda>0$ in $B^c_R$. Therefore, if we prove that there exist large $R>1$ such that 
\begin{align*}
Q_1>0 \text{ for all } y \in B^c_R 
\end{align*}
where $Q_1:= -\frac{1}{2} q (q-1) \phi^{q-2} |D \phi|^2 + (q-1)\phi^{q-1} (f-\lambda)$, we would have $Q>0 \text{ for all } y \in B^c_R $.

We see that, since $q>1$ and $\phi \geq 1$, $Q_1>0$ is equivalent to
$$\frac{1}{2} q \frac{|D \phi|^2}{\phi}  < f-\lambda\; .$$
By Proposition~\ref{bainf1}, Assumption $(H1)$ and the fact that $\phi \in \Phi_\gamma$ we can see that there are constants $K,M>0$ such that
\begin{align*}
\frac{1}{2} q \frac{|D \phi(y)|^2}{\phi(y)} \leq K |y|^{\gamma-2}.
\end{align*}
and 
\begin{equation*}
 f(y)-\lambda \geq M |y|^\alpha 
\end{equation*}
for all $y$ in the complementary of a (possible large) ball $B_R$. Therefore, to have inequality ${Q}_1>0$ (at least for $R$ large) it is enough to have $\gamma-2 <\alpha$. But
\begin{equation*}
\alpha-(\gamma-2) =\alpha-\frac{\alpha}{\theta}+1=\frac{\alpha}{\theta^*}+1>0.
\end{equation*}
Consequently, there exist a $R>1$ such that $Q_1>0$ for all $y \in B^c_R$. It is worth remarking that such $R$ is independent of $q \in (1, q_0)$.
\end{proof}
\begin{proposition}\label{comp-subq}
Suppose that $\phi$ and $\psi$ are respectively a subsolution and a supersolution of $(EP)$ such that $\phi, \psi \in \Phi_\gamma$, $\psi \geq 1$ and $\phi \leq \psi$ on $\partial B_R $. If $R>1$ is as in Lemma~\ref{lem-subq}, then $\phi \leq \psi^{q}$ in $B^c_R$.
\end{proposition}
\begin{proof}
By Lemma~\ref{lem-subq}, we know that $(\lambda, \psi^q)$ is a strict supersolution of $(EP)$ in $B^c_R$. We wish to prove that $\phi \leq \psi^{q}$ in $B^c_R$.

We first notice that, since $\phi, \psi \in \Phi_\gamma$, $\phi$ grows like $|y|^\gamma$ at infinity while $\psi^{q} $ grows like $|y|^{q\gamma}$. Since $q>1$, we can conclude
\begin{equation*}
(\phi-\psi^q)(y) \to -\infty \text{ as } |y| \to \infty.
\end{equation*}
Hence there exists a maximum point $y^* \in B^c_R$ of $\phi-\psi^q$: if $(\phi-\psi^q)(y^*) \leq 0$, we are done. Therefore we can assume that $(\phi-\psi^q)(y^*) > 0$.

If $y^* \in \partial B_R$, $\phi(y^*) \leq \psi (y^*) \leq \psi^q(y^*)$ because $\psi \geq 1$ and then we would have that $\phi(y^*)-\psi^q(y^*) \leq 0$ a contradiction. Therefore $|y^*|>R$ and we have
$$
D \phi = D \psi^q
\quad\hbox{and}\quad
\Delta (\phi - \psi^q) \leq 0.
$$
We then arrive at
\begin{align*}
f(y^*)-\lambda \geq -\frac{1}{2} \Delta \phi (y^*)+\frac{1}{\theta} |D \phi(y^*)|^\theta \geq & -\frac{1}{2} \Delta \psi^q (y^*)+\frac{1}{\theta} |D \psi^q(y^*)|^\theta\\
 >& f(y^*)-\lambda \text{ in } B^c_R 
\end{align*}
a contradiction. Therefore, $\phi \leq \psi^q$ in $B^c_R$.
\end{proof}
\begin{corollary}
Suppose that $\phi$ and $\psi$ are, respectively, a subsolution and a supersolution of $(EP)$ such that $\phi, \psi \in \Phi_\gamma$, $\psi \geq 1$ and $\phi \leq \psi$ on $\partial B_R $. If $R>1$ is as in Lemma~\ref{lem-subq}, then $\phi \leq \psi$ in $B^c_R$.
\end{corollary}
\begin{proof}
Since $R$ in Lemma~\ref{lem-subq} is independent of $q$, the conclusion follows by letting $q \to 1$ in Proposition~\ref{comp-subq}.
\end{proof}
\begin{theorem}
Let $\theta <2$ and suppose that $(\lambda_1,\phi)$ and $(\lambda_2,\psi)$ are two solutions of $(EP)$ such that $\phi, \psi \in \Phi_\gamma$. Then, $\phi=\psi+C$ and $\lambda_1=\lambda_2$.
\end{theorem}
\begin{proof}
Suppose that $\lambda_1 \geq \lambda_2$. Otherwise we exchange the roles of $\phi$ and $\psi$ in the argument. Notice that if $\lambda_1 \geq \lambda_2$, then $(\lambda_1,\phi)$ is a subsolution of $(EP)$ with $\lambda=\lambda_2$.

We saw in Step 1 of Theorem 3.1 that for a fixed $R>1$ we can always add constants to $\phi$ and $\psi$ and ask that $\max_{\partial B_R} (\phi-\psi) =0$ and $\psi \geq 1$. We now look at $R>1$ given by Lemma~\ref{lem-subq} for which we know that $(\lambda_2,(\psi)^q)$ for $q>1$ is a strict supersolution of $(EP)$ with $\lambda=\lambda_2$. Corollary 3.8 give us now $\phi \leq \psi$ in $B^c_R$. From another point of view, $\max_{\partial B_R} (\phi-\psi) =0$ implies a comparison in the ball, $\phi \leq \psi$ in $\bar{B}_R$. Therefore, $\phi \leq \psi$ in $\mathbb{R}^m$ and we can repeat the argument of Theorem~\ref{uni-sup} to conclude that $\phi-\psi=C$ is a constant in $\mathbb{R}^m$. Consequently, from the equation of $(EP)$, $\lambda_1=\lambda_2$.
\end{proof}
Since solutions of $(EP)$ that are bounded from below belong to $\Phi_\gamma$ (Proposition 3.4), we have
\begin{theorem} (Uniqueness result)
Let $\theta <2$ and suppose that $(\lambda_1,\phi)$ and $(\lambda_2,\psi)$ are two solutions of $(EP)$ bounded from below. Then, $\phi=\psi+C$ and $\lambda_1=\lambda_2$.
\end{theorem}

\section{Consequence of the uniqueness results : properties of $\lambda^*$}\label{l-star}

Theorem~\ref{exist-CV} shows the existence of a critical value
\begin{equation*} 
\lambda^* := \sup \{\lambda \in \mathbb{R} | \text{ $(EP)$ has a subsolution} \}
\end{equation*}
such that $(EP)$ admits a classical subsolution $\phi \in C^2 (\mathbb{R}^m)$ if and only if $\lambda \leq \lambda^*$. The following result gives a caracterisation of $\lambda^*$.

\begin{proposition}\label{uni-l-star} Under either the assumptions of Proposition~\ref{comp-ext} or \ref{comp-subq}, if $(\lambda,\phi)$ is solution of $(EP)$with $\phi$ bounded from below. Then $\lambda=\lambda^*$.
\end{proposition} 
\begin{proof} Let $\psi$ be a solution of $(EP)_{\lambda^*}$ (see Theorem~\ref{exist-CV}). We know that $\lambda \leq \lambda^*$. It remains to show that $\lambda^* \leq \lambda$. 

We notice that $\psi$ is a subsolution for $(EP)_\lambda$. Choosing $R$ large enough, we may assume that $\max_{\partial{B}_R}(\psi-\phi)=0$ and then consider $\chi:= \max(\psi,\phi)$ which is still a bounded from below subsolution of $(EP)_\lambda$ with $\max_{\partial{B}_R}(\chi-\phi)=0$. 

We deduce from one of the comparison results in $B^c_R$ (either the sub or superquadratic one) that $\chi \leq \phi$ in $\R^m$, i.e. $\psi \leq \phi$ in $\R^m$. But by the Strong Maximum Principle, $\psi-\phi$ achieving a maximum at a point of $\partial{B}_R$, we conclude that $\psi=\phi$ in $\R^m$ and that $\lambda =\lambda^*$.
\end{proof}

\subsection{Convergence of approximations, large time behavior}

Our first result is the

\begin{proposition}\label{conv-app}Under the assumptions of Proposition~\ref{uni-l-star}, we have the following

(i) Let $(\phi^R,\lambda_R)$ be, as in the proof of Theorem~\ref{exist-bfb}, a solution of the ergodic problem in $B_R$ with state-constraints boundary conditions and with $\phi^R(0)=0$. Then, as $R\to +\infty$, $(\phi^R,\lambda_R)$ converges to $(\phi,\lambda^*(f))$ where $\phi$ is a solution of $G[\phi]=\lambda^*(f)$ in $\R^m$.\\

(ii) Let, for $R\gg 1$, $(\psi_R,\tilde \lambda_R)$ be a solution of an ergodic problem associated to $f_R \in W^{1,\infty}_{loc}(\R^m)$ where the $f^R$ are uniformly bounded in $W^{1,\infty}_{loc}$ and which, in the subquadratic case, satisfy (H0) with uniform constants. If $f_R\to f$ locally uniformly where $f$ satisfies (H0)-(H1) and if the $\phi_R$ are uniformly bounded from below, then $\tilde \lambda_R\to \lambda^*(f)$ and $\phi_R-\phi_R(0)$ converges to $\phi$ where $\phi$ is a solution of $G[\phi]=\lambda^*(f)$ in $\R^m$.
\end{proposition}

We leave the easy proof of this result to the reader and turn to the application to the large time behavior of solutions of the associated parabolic equation
\begin{equation}\label{parabolic-eqn}
u_t - \frac{1}{2}\Delta u+\frac{1}{\theta} | D u|^{\theta}=f  \quad \text{in } \mathbb{R}^m \times (0,+\infty)\; ,
\end{equation}
\begin{equation}\label{parabolic-id}
u(x,0)=u_0(x)  \quad \text{in } \mathbb{R}^m\; ,
\end{equation}
where $u_0$ is locally Lipschitz continuous coercive function.

\begin{theorem} Under the assumptions of Proposition~\ref{uni-l-star} and if $u_0$ is locally Lipschitz continuous coercive function, which satisfies in the subquadratic case
$$ c_0|y|^\gamma-c_1 \leq u_0(y) \leq c_1(1+|y|^\gamma) \quad \text{in } \mathbb{R}^m\; ,$$
for some constants $c_0,c_1>0$, there exists a unique, global in time, solution $u$ of (\ref{parabolic-eqn})-(\ref{parabolic-id}) which satisfies
$$ \lim_{t\to +\infty}\, \left[\frac{u(x,t)}{t}\right]=\lambda^*(f)\quad \text{ locally uniformly in } \mathbb{R}^m\; .$$
\end{theorem}

\begin{proof} The existence and uniqueness properties are obtained by borrowing the arguments we already use for proving the existence and uniqueness properties for the solutions of the ergodic problem, so we skip them. 

To prove the ergodic limit property, we argue by approximation : if $(\phi^R,\lambda_R)$ is a solution of the ergodic problem in $B_R$ with state-constraints boundary conditions, we have by the comparison results of Tchamba Tabet \cite{T10} in the superquadratic case or Porretta, Tchamba and the first author\cite{BPT10} in the subquadratic case [adding, if necessary a constant to $\phi^R$]
$$ u(x,t) \leq \phi^R (x) + \lambda_R t \quad \text{in } \mathbb{R}^m \times (0,+\infty)\; ,$$
since $u$ is a subsolution of the parabolic state constraint problem in $B_R \times (0,+\infty)$ and $\phi^R (x) + \lambda_R t$ is a solution.
Therefore
$$ \limsup_{t\to +\infty}\, \left[\frac{u(x,t)}{t}\right]\leq \lambda_R \quad \text{ locally uniformly in } \mathbb{R}^m\; .$$
And since this is true for any $R$, we conclude that the same inequality holds for $\lambda^*(f)$ by Proposition~\ref{conv-app}.

It remains to prove the opposite inequality with the $\liminf$ and, to do so, we approximate $f$ from below by $f_R$ defined in the following way: since $f$ is coercive, $\min(f,R)(y)$ is equal to $R$ for $|y|$ large enough, say for $|y|\geq S_R$ and clearly $S_R\to +\infty$ as $R\to +\infty$. We define $f_R$ as being the $2S_R\Z^m$-periodic function which is equal to $\min(f,R)$ on $[-S_R,S_R]^m$. By results in the periodic case (see \cite{bs01} and references therein), there exists a solution $(\psi_R,\tilde \lambda_R)$ of the ergodic problem associated to $f_R$ and by (easier) comparison results, we have
$$ u(x,t) \geq \psi_R (x) + \tilde \lambda_R t \quad \text{in } \mathbb{R}^m \times (0,+\infty)\; ,$$
shifting again the periodic function $\psi_R (x)$ in order to have $\psi_R (x) \leq u_0(x)$ in $\R^m$.

Therefore
$$ \liminf_{t\to +\infty}\, \left[\frac{u(x,t)}{t}\right]\geq \tilde \lambda_R \quad \text{ locally uniformly in } \mathbb{R}^m\; .$$
And since this is true for any $R$, we conclude that the same inequality holds for $\lambda^*(f)$ by Proposition~\ref{conv-app}.
\end{proof}

\subsection{Continuity properties}

In applications to singular perturbation type problems, as it is the case for homogenization type problems, the continuity properties of $\lambda^*(f)$ in $f$ are important. The next result refines these continuity properties.

\begin{proposition} Assume that $f_1,f_2$ satisfies (H0)-(H1) with $\alpha \geq 1$ and with a fixed $f_0$ and set
$$ m := \sup_{\R^m}(\frac{|f_1(y)-f_2(y)|}{1+|y|^\alpha})\; .$$
We have 
$$|\lambda^*(f_2)-\lambda^*(f_1)|\leq \frac{f_0m}{1+f_0m}\max(\lambda^*(f_1) ,\lambda^*(f_2))\; .$$
\end{proposition}
\begin{proof} By definition of $m$, we have
$$ f_1(y)\leq f_2(y)+ m(1+|y|^\alpha)\; .$$
Let $\phi_1$ be a solution of $(EP)_{\lambda^*}$ associated to $f_1$. If we consider $t\phi_1$ for $t\in(0,1)$, it is a subsolution of the equation with a right hand side with is $t(f_1-\lambda^*(f_1))$. On the other hand
$$
t f_1(y)\leq t f_2(y)+ tm(1+|y|^\alpha)\leq t(1+f_0m)f_2(y) \; .$$
Choosing $t=(1+f_0m)^{-1}$, we deduce that $t\phi_1$ is subsolution with right-hand side $f_2 (y)-t\lambda^*(f_1)$ and therefore
$$ t\lambda^*(f_1) \leq \lambda^*(f_2)\; .$$
From which we deduce
$$ (t-1)\lambda^*(f_1) \leq \lambda^*(f_2)-\lambda^*(f_1)\; ,$$
i.e.
$$ -\frac{f_0m}{1+f_0m}\lambda^*(f_1) \leq \lambda^*(f_2)-\lambda^*(f_1)\; ,$$
and reversing the role of $f_1,f_2$, the inequality we wanted to prove.
\end{proof}

\begin{Remark}Of course, if we want to take into account function $f$ which satisfy a little bit more general assumption than (H1) (in particular, for the lower bound), we have just to add a suitable constant to $f_1$ and $f_2$, $\lambda^*(f_1), \lambda^*(f_2)$ being shifted by the same constant according to Proposition~\ref{lambda-fp}.
\end{Remark}
Finally in order to have a (uniform in $f_0$) estimate of the $\lambda^*(f_i)$, $i=1,2$ in the previous result, we use the
\begin{proposition} Let $c>0$. If $\alpha \geq 1$, then
\begin{equation*} 
\lambda^*(c|y|^\alpha)=c^{\frac{\theta^*}{\theta^*+\alpha}} \lambda^*(|y|^\alpha).
\end{equation*}
If $0<\alpha<1$, then
\begin{equation*}
0 \leq \lambda^*(c(1+|y|^2)^\frac{\alpha}{2}) \leq c+c^{\frac{\theta^*}{\theta^*+1}} \lambda^*(|y|).
\end{equation*}
\end{proposition}
\begin{proof} 
We first observe that $|y|^\alpha$ satisfy assumption $(H0)$ when $\alpha \geq 1$ and is clearly nonnegative. Let $(\lambda^*(|y|^\alpha),\phi_1)$ be a solution of $(EP)$ with $f(y)=|y|^\alpha$ given by Theorem~\ref{exist-CV}. We will now construct a solution of $(EP)$ with $f(y)=c|y|^\alpha$ by considering $\phi_2(y)=\beta^{\frac{2-\theta}{\theta-1}}\phi_1(\beta y)$ and the right choice of $\beta$.

We have,
\begin{align*}
-\frac{1}{2} \Delta \phi_2(y)+\frac{1}{\theta} |D \phi_2(y)|^\theta &= -\frac{1}{2} \Delta (\beta^{\frac{2-\theta}{\theta-1}}\phi_1(\beta y))+\frac{1}{\theta} |D (\beta^{\frac{2-\theta}{\theta-1}}\phi_1(\beta y))|^\theta \\
&=-\frac{1}{2} \beta^{\frac{\theta}{\theta-1}} \Delta (\phi_1(\beta y))+\frac{1}{\theta} \beta^{\frac{\theta}{\theta-1}} |D \phi_1(\beta y)|^\theta\\
&=\beta^{\frac{\theta}{\theta-1}}  \big(-\frac{1}{2} \Delta (\phi_1(\beta y))+\frac{1}{\theta} |D \phi_1(\beta y)|^\theta \big)\\
&=\beta^{\theta^*} \big( \beta^\alpha |y|^\alpha - \lambda^*(|y^\alpha|) \big)
\end{align*}
using the chain rule for the second equality and the fact that $(\lambda^*(|y|^\alpha),\phi_1)$ is a solution of $(EP)$ with $f(y)=|y|^\alpha$ and $\frac{\theta}{\theta-1}=\theta^*$ for the last.

Therefore
\begin{align*}
\beta^{\theta^*} \lambda^*(|y^\alpha|) -\frac{1}{2} \Delta \phi_2(y)+\frac{1}{\theta} |D \phi_2(y)|^\theta =\beta^{\theta^*+\alpha} |y|^\alpha
\end{align*}
and choosing $\beta^{\theta^*+\alpha}=c$, i.e., $\beta=c^{\frac{1}{\theta^*+\alpha}}$ we arrive at
\begin{align*}
c^{\frac{\theta^*}{\theta^*+\alpha}} \lambda^*(|y^\alpha|) -\frac{1}{2} \Delta \phi_2(y)+\frac{1}{\theta} |D \phi_2(y)|^\theta =c |y|^\alpha.
\end{align*}
By definition of $\lambda^*(c |y|^\alpha)$, we obtain
\begin{equation*}
c^{\frac{\theta^*}{\theta^*+\alpha}} \lambda^*(|y^\alpha|) \leq \lambda^*(c |y|^\alpha).
\end{equation*}
The reverse inequality is obtained in an equivalent manner by looking at the solution $(\lambda^*(c|y|^\alpha),\psi_1)$ of $(EP)$ with $f(y)=c|y|^\alpha$ and then constructing a solution of $(EP)$ with $f(y)=|y|^\alpha$ by considering $\psi_2(y)=\beta^{\frac{2-\theta}{\theta-1}}\psi_1(\beta y)$ and $\beta=(\frac{1}{c})^{\frac{1}{\theta^*+\alpha}}$. \\

In the case when $0<\alpha<1$, we look at $c(1+|y|^2)^{\frac{\alpha}{2}}$ which is in $W^{1,\infty}_{loc}(\mathbb{R}^m)$ and is non-negative. By the proof of Theorem~\ref{exist-CV}, $\lambda^*(c(1+|y|^2)^{\frac{\alpha}{2}}) \geq 0$, and since 
\begin{equation*}
c(1+|y|^2)^{\frac{\alpha}{2}} \leq c(1+|y|^\alpha) \leq c(1+|y|),
\end{equation*}
using Propositions 2.9, 2.8 and the first part of this proof with $\alpha \geq 1$, 
\begin{equation*}
0 \leq \lambda^*(c(1+|y|^2)^\frac{\alpha}{2}) \leq c+c^{\frac{\theta^*}{\theta^*+1}} \lambda^*(|y|)
\end{equation*}
as we wished to show.
\end{proof}

\begin{appendices}

\section{- Gradient Estimate}

In this Appendix, we present some results and estimates needed in this article. We start with the following result 
\begin{theorem} For any $R>0$, $f_1 \in W^{1,\infty}(B_R)$ and $g_1 \in C^{2,\iota} (\partial B_R)$ where $\iota \in (0,1)$. Then,
\begin{itemize}
\item[(a)] For any $\epsilon > 0$, the Dirichlet problem
\begin{equation*}
-\frac{1}{2} \Delta \phi (y)+ \frac{1}{\theta} |D \phi(y)|^{\theta}+ \epsilon \phi = f_1 \text{ in } B_R, \quad \phi = g_1 \text{ on } \partial B_R, 
\end{equation*}
has a $C^{2,\iota} (\bar{B}_R)$-solution.
\item[(b)] The Dirichlet problem
\begin{equation*}
-\frac{1}{2} \Delta \phi (y)+ \frac{1}{\theta} |D \phi(y)|^{\theta} = f_1 \text{ in } B_R, \quad \phi = g_1 \text{ on } \partial B_R, 
\end{equation*}
has a $C^{2,\iota} (\bar{B}_R)$-solution provided it has a subsolution which is in $C^2(B_R) \cap C(\bar{B}_R) $.
\end{itemize}
\end{theorem}
\begin{proof}
Claim (a) is a particular case of results of \cite{Lions}. Claim (b) can be found in Theorem A.1 of \cite{Lions} and it uses the convexity of the operator $I$ and Theorem 6.14 of \cite{GT}.
\end{proof}

The following result appears in \cite{NI3} (see also \cite{LaLi}).
\begin{theorem}
Let $\Omega$ and $\Omega'$ be two bounded open sets in $\mathbb{R}^m$ such that $\bar{\Omega}' \subset \Omega$. For given $\epsilon \in [0, 1)$ and $f_1 \in W^{1,\infty}_{\text{loc}}(\mathbb{R}^m)$, if $\phi \in C^2(\mathbb{R}^m)$ is a solution of the elliptic equation
\begin{equation}\label{Basice}
-\frac{1}{2} \Delta \phi +\frac{1}{\theta} |D\phi|^\theta + \epsilon \phi = f_1 \text{ in } \Omega\; ,
\end{equation}
then, there exists a constant $K> 0$ depending only on $m$, $\theta$ and $\text{dist}(\Omega',\partial \Omega)$ such that
\begin{equation*}
\sup_{\Omega'} |D\phi| \leq K (1 + \sup_{\Omega} |\epsilon \phi|^{\frac{1}{\theta}} + \sup_{\Omega} |f_1|^{\frac{1}{\theta}}+ \sup_{\Omega} |Df_1|^{\frac{1}{2\theta-1}}).
\end{equation*}
In particular, in the case when $\Omega=B_{R}$ and  $\Omega'=B_{R'}$ for some $R\geq R'+1>0$, there exists $K>0$ depending only on $m$ and $\theta>1$ such that 
\begin{equation*}
\sup_{B_{R'}} |D \phi| \leq K(1+\sup_{B_{R}}|f_1|^{\frac{1}{\theta}}+\sup_{B_{R}}|Df_1|^{\frac{1}{2\theta-1}}).
\end{equation*}
for any $C^2$-solution of Equation~(\ref{Basice}).
\end{theorem}
\end{appendices}


\begin{thebibliography}{10}


\bibitem{gl05}
{\sc G.~Barles and F.~Da~Lio}, {\em On the boundary ergodic problem for fully
  nonlinear equations in bounded domains with general nonlinear {N}eumann
  boundary conditions}, Ann. Inst. H. Poincar\'e Anal. Non Lin\'eaire, 22
  (2005), pp.~521--541.

\bibitem{BLLS08}
{\sc G.~Barles, F.~Da~Lio, P.-L. Lions, and P.~E. Souganidis}, {\em Ergodic
  problems and periodic homogenization for fully nonlinear equations in
  half-space type domains with {N}eumann boundary conditions}, Indiana Univ.
  Math. J., 57 (2008), pp.~2355--2375.

\bibitem{BIM12}
{\sc G.~Barles, H.~Ishii, and H.~Mitake}, {\em On the large time behavior of
  solutions of {H}amilton-{J}acobi equations associated with nonlinear boundary
  conditions}, Arch. Ration. Mech. Anal., 204 (2012), pp.~515--558.

\bibitem{BM12}
{\sc G.~Barles and H.~Mitake}, {\em A {PDE} approach to large-time asymptotics
  for boundary-value problems for nonconvex {H}amilton-{J}acobi equations},
  Comm. Partial Differential Equations, 37 (2012), pp.~136--168.

\bibitem{BPT10}
{\sc G.~Barles, A.~Porretta, and T.~T. Tchamba}, {\em On the large time
  behavior of solutions of the {D}irichlet problem for subquadratic viscous
  {H}amilton-{J}acobi equations}, J. Math. Pures Appl. (9), 94 (2010),
  pp.~497--519.

\bibitem{BR}
{\sc G.~Barles and J.-M.~Roquejoffre}, {\em Ergodic type problems and large time behaviour of unbounded solutions of Hamilton-Jacobi equations}. Comm. Partial Differential Equations 31 (2006), no. 7-9, 1209–1225. 


\bibitem{BS00a}
{\sc G.~Barles and P.~E. Souganidis}, {\em Some counterexamples on the
  asymptotic behavior of the solutions of {H}amilton-{J}acobi equations}, C. R.
  Acad. Sci. Paris S\'er. I Math., 330 (2000), pp.~963--968.

\bibitem{bs01}
{\sc G.~Barles and P.~E. Souganidis}, {\em Space-time periodic solutions and
  long-time behavior of solutions to quasi-linear parabolic equations}, SIAM J.
  Math. Anal., 32 (2001), pp.~1311--1323 (electronic).
  
\bibitem{Ci}
{\sc M. Cirant}, {\em  On the solvability of some ergodic control problems in $\R^d$}. SIAM J. Control Optim. 52 (2014), no. 6, 4001–4026. 

\bibitem{L08}
{\sc F.~Da~Lio}, {\em Large time behavior of solutions to parabolic equations
  with {N}eumann boundary conditions}, J. Math. Anal. Appl., 339 (2008),
  pp.~384--398.

\bibitem{DS}
{\sc A.~Davini and A.~Siconolfi}, {\em A generalized dynamical approach to the
  large time behavior of solutions of {H}amilton-{J}acobi equations}, SIAM J.
  Math. Anal., 38 (2006), pp.~478--502 (electronic).

\bibitem{F12}
{\sc A.~Fathi}, {\em Weak KAM from a PDE point of view: viscosity solutions of the Hamilton-Jacobi equation and Aubry set}.
Proc. Roy. Soc. Edinburgh Sect. A 142 (2012), no. 6, 1193–1236. 

\bibitem{GT}
{\sc D. Gilbarg and N.-S. Trudinger}, {\em Elliptic Partial Differential Equations of
Second Order}, Springer, 2001.

\bibitem{NI1}
{\sc N. Ichihara}, {\em
The generalized principal eigenvalue for Hamilton-Jacobi-Bellman equations of ergodic type}. Ann. Inst. H. Poincar\'e Anal. Non Lin\'eaire 32 (2015), no. 3, 623–650. 

\bibitem{NI2}
{\sc N. Ichihara}, {\em
Criticality of viscous Hamilton-Jacobi equations and stochastic ergodic control}.
J. Math. Pures Appl. (9) 100 (2013), no. 3, 368–390. 

\bibitem{NI3}
{\sc N. Ichihara}, {\em Large time asymptotic problems for optimal stochastic control with superlinear cost}. Stochastic Process. Appl. 122 (2012), no. 4, 1248–1275. 
 
\bibitem{NI4}
{\sc N. Ichihara}, {\em
Recurrence and transience of optimal feedback processes associated with Bellman equations of ergodic type}. SIAM J. Control Optim. 49 (2011), no. 5, 1938–1960. 

\bibitem{I08}
{\sc H.~Ishii}, {\em Asymptotic solutions for large time of {H}amilton-{J}acobi
  equations in {E}uclidean {$n$} space}, Ann. Inst. H. Poincar\'e Anal. Non
  Lin\'eaire, 25 (2008), pp.~231--266.

\bibitem{I11}
{\sc H.~Ishii}, {\em Long-time asymptotic
  solutions of convex {H}amilton-{J}acobi equations with {N}eumann type
  boundary conditions}, Calc. Var. Partial Differential Equations, 42 (2011),
  pp.~189--209.

\bibitem{LaLi}
{\sc J.-M. Lasry and P.-L.~Lions, P.-L.}, {\em
Nonlinear elliptic equations with singular boundary conditions and stochastic control with state constraints}. I. The model problem. 
Math. Ann. 283 (1989), no. 4, 583–630.

\bibitem{Lions}
{\sc P.-L. Lions}, {\em Generalized solutions of {H}amilton-{J}acobi
  equations}, vol.~69 of Research Notes in Mathematics, Pitman (Advanced
  Publishing Program), Boston, Mass., 1982.

\bibitem{lpv}
{\sc P.-L. Lions, G.~Papanicolaou, and S.~Varadhan}, {\em Homogenization of
  {H}amilton-{J}acobi equations}.
\newblock Preprint, 1987.

\bibitem{M08}
{\sc H.~Mitake}, {\em Asymptotic solutions of {H}amilton-{J}acobi equations
  with state constraints}, Appl. Math. Optim., 58 (2008), pp.~393--410.


\bibitem{T10}
{\sc T.~Tabet~Tchamba}, {\em Large time behavior of solutions of viscous
  {H}amilton-{J}acobi equations with superquadratic {H}amiltonian}, Asymptot.
  Anal., 66 (2010), pp.~161--186.


\end{thebibliography}
\end{document}